\documentclass{amsart}

\usepackage{amsmath}
\usepackage{amssymb}
\usepackage{amscd}
\usepackage{hyperref}
\newtheorem{theorem}{Theorem}[section]
\newtheorem{thm}[theorem]{Theorem}
\newtheorem{cor}[theorem]{Corollary}
\newtheorem{lem}[theorem]{Lemma}
\newtheorem{prop}[theorem]{Proposition}

\theoremstyle{definition}
\newtheorem{defn}[theorem]{Definition}

\newtheorem{rem}[theorem]{Remark}

\newtheorem{ex}[theorem]{Example}

\newtheorem{conj}[theorem]{Conjecture}

\theoremstyle{remark}

\newcommand{\mbb}{\mathbb}
\newcommand{\QQ}{\mbb{Q}}

\newcommand{\ZZ}{\mbb{Z}}

\newcommand{\PP}{\mbb{P}}

\newcommand{\mc}{\mathcal}

\newcommand{\mcX}{\mc{X}}

\newcommand{\OO}{\mc{O}}

\newcommand{\SP}{\text{Spec }}

\newsavebox{\sembox}
\newlength{\semwidth}
\newlength{\boxwidth}

\newcommand{\Sem}[1]{%
\sbox{\sembox}{\ensuremath{#1}}%
\settowidth{\semwidth}{\usebox{\sembox}}%
\sbox{\sembox}{\ensuremath{\left[\usebox{\sembox}\right]}}%
\settowidth{\boxwidth}{\usebox{\sembox}}%
\addtolength{\boxwidth}{-\semwidth}%
\left[\hspace{-0.3\boxwidth}%
\usebox{\sembox}%
\hspace{-0.3\boxwidth}\right]%
}

\newsavebox{\semrbox}
\newlength{\semrwidth}
\newlength{\boxrwidth}

\newcommand{\Semr}[1]{%
\sbox{\semrbox}{\ensuremath{#1}}%
\settowidth{\semrwidth}{\usebox{\semrbox}}%
\sbox{\semrbox}{\ensuremath{\left(\usebox{\semrbox}\right)}}%
\settowidth{\boxrwidth}{\usebox{\semrbox}}%
\addtolength{\boxrwidth}{-\semrwidth}%
\left(\hspace{-0.3\boxrwidth}%
\usebox{\semrbox}%
\hspace{-0.3\boxrwidth}\right)%
}

\title[Zero cycles on rationally connected varieties]
{Zero cycles on rationally connected varieties over Laurent fields}

\author[Tian]{Zhiyu Tian}
\address{
Beijing International Center for Mathematical Research\\
Peking University\\
100871, Beijing, China}
\email{zhiyutian@bicmr.pku.edu.cn}

\date{\today}

\begin{document}

%%%%%%%%%%%%%%%%%%%%%%%%%%%%%%%%%%%%%%%%%%%%%%%%%%%%%%%%%%%%%%%%%%%%
%%
%% Abstract
%%
%%%%%%%%%%%%%%%%%%%%%%%%%%%%%%%%%%%%%%%%%%%%%%%%%%%%%%%%%%%%%%%%%%%%

\begin{abstract}
We study zero cycles on rationally connected varieties defined over characteristic zero Laurent fields with algebraically closed residue fields. We show that the degree map induces an isomorphism for rationally connected threefolds defined over such fields. In general, the degree map is an isomorphism if  rationally connected varieties defined over algebraically closed fields of characteristic zero satisfy the integral Hodge/Tate conjecture for one cycles, or if the Tate conjecture is true for divisor classes on surfaces defined over finite fields. To prove these results, we introduce techniques from the minimal model program to study the homology of certain complexes defined by Kato/Bloch-Ogus.
\end{abstract}

%%%%%%%%%%%%%%%%%%%%%%%%%%%%%%%%%%%%%%%%%%%%%%%%%%%%%%%%%%%%%%%%%%%%%%
%%
%% Body
%%
%%%%%%%%%%%%%%%%%%%%%%%%%%%%%%%%%%%%%%%%%%%%%%%%%%%%%%%%%%%%%%%%%%%%%%

\maketitle

%% \tableofcontents

%%%%%%%%%%%%%%%%%%%%%%%%%%%%%%%%%%%%%%%%%%%%%%%%%%%%%%%%%%%%%%%
%%
%% Section: Introduction
%%
%%%%%%%%%%%%%%%%%%%%%%%%%%%%%%%%%%%%%%%%%%%%%%%%%%%%%%%%%%%%%%%
\section{Introduction}
A variety $X$ is called \emph{rationally connected} if for any two general geometric point $x, y \in X(\Omega)$, there is a morphism $f: \PP^1_\Omega \to X_\Omega$ such that $f(0)=x, f(\infty)=y$.
 It follows from this definition that the degree map for the Chow group of zero cycles on a smooth projective rationally connected variety defined over an algebraically closed field induces an isomorphism to $\ZZ$. 
 But for rationally connected variety defined over other fields, this group remains mysterious, except that the kernal of the degree map is a torsion group of some exponent $N$ (depending on $X$).

 For rationally connected varieties defined over an arithmetically interesting field, one could like to understand this kernel.
 As an application of higher dimensional class field theory of Kato and Saito \cite{KatoSaito_CFT}, we know the kernel is trivial for separably rationally connected varieties over finite fields.
 This result is also a consequence of Koll\'ar-Szab\'o's study of fundamental groups of separably rationally connected varieties \cite{KollarSzabo}.
 For surfaces, since $0$-cycles are codimension $2$ cycles, one can use K-theory to study them (See \cite{CT_K_cycle} for an account).
 This approach shows that over a field of cohomological dimension $1$, the kernel is trivial (\cite{CTChow}).
 Building on work of Saito-Sato \cite{SaitoSato_0_cycle}, Esnault and Wittenberg \cite{WittenbergEsnault_0_cycle} studied the kernel of the cycle class map for Chow group of zero cycles to the \'etale cohomology for smooth projective varieties defined over a henselian discrete valuation field with finite or separably closed residue field.
 In particular, they proved that the cycle class map is injective for rational surfaces over such fields.
We refer the readers to the surveys \cite{CT_RCSurvey,WittenbergRCinArith} for a discussion of this question (among others).

 In this article, we introduce techniques from the minimal model program (MMP) to study this question. The main theorem is the following:
\begin{thm}\label{thm:Laurant}
Let $k$ be an algebraically closed field of characteristic $0$ and let $X$ be a smooth projective rationally connected variety defined over $k\Semr{t}$. 
Assume that one of the following holds.
\begin{enumerate}
\item $\dim X \leq 3$.
\item $\dim X=d$ and the cycle class map $$CH_1(Y)/n \to H^{2d-2}_\text{\'et}(Y, \ZZ/n\ZZ(d-1))$$ is surjective for any smooth projective rationally connected variety $Y$ defined over $k$ of dimension at most $d$.
\item The Tate conjecture (Conjecture \ref{conj:Tate}) holds for all surfaces defined over a finite field.
\end{enumerate}
Then the degree map induces an isomorphism $$\deg: CH_0(X)\cong \ZZ.$$
\end{thm}
\begin{rem}
We have to restrict ourselves to characteristic $0$ Laurent fields because one of the essential ingredients in the proof is the minimal model program (MMP), which is not sufficiently developed in positive and mixed characteristics in higher dimensions. However, it is expected that MMP should also work in these cases. Thus our result suggests that the same conclusion should hold for separably rationally connected varieties defined over the fraction field of a henselian local ring whose residue field is an algebraic closure of a finite field.
\end{rem}

Next we explain the strategy for the proof of Theorem \ref{thm:Laurant} and some intermediate result that should be of independent interest. 
By the result of Saito-Sato \cite{SaitoSato_0_cycle}, a sufficient condition for the vanishing of the kernel of the degree map is the following (for a complete description of the kernel, see \cite[Theorem 2.1, 2.2]{WittenbergEsnault_0_cycle}):

\begin{thm}\label{thm:IHC_degeneration}
Let $X$ a smooth projective rationally connected variety defined over $k\Semr{t}$, and let $\mcX \to \SP k \Sem{t}$ be a projective regular model of $X$ with simple normal crossing central fiber $\mcX_0$. Assume one of the followings holds.
\begin{enumerate}
\item $\dim X \leq 3$.
\item $\dim X=d$ and the cycle class map $$CH_1(Y)/n \to H^{2d-2}_\text{\'et}(Y, \ZZ/n\ZZ(d-1))$$ is surjective for any smooth projective rationally connected variety $Y$ defined over $k$ of dimension at most $d$.
\item The Tate conjecture (Conjecture \ref{conj:Tate}) holds for all surfaces defined over a finite field.
\end{enumerate}
Then $CH_1((\mcX_0)_{\text{red}})/n \to H_2((\mcX_0)_{\text{red}}, \ZZ/n\ZZ(-1))$ is surjective. Here $H_2(\bullet, \ZZ/n\ZZ)$ is the \'etale Borel-Moore homology.
\end{thm}

This result is about an integral Hodge/Tate conjecture for degenerations of rationally connected varieties. 
The main observation of this article is that such a conjecture (and more general results about cycles) for degenerations of rationally connected varieties follows from the integral Hodge/Tate conjecture (and corresponding results) for smooth rationally connected varieties.

To further explain this observation, we need the concept of Kato homology for a variety. The precise definition can be found in Section \ref{subsec-Kato}. For the moment, it suffices to know that the $a$-th Kato homology with coefficients in an abelian group $A$, denoted by $KH_a(X, A)$, is defined for all finite type schemes over a field with a homology theory with coefficient $A$. The Borel-Moore homology with $\ZZ$-coefficient for complex varieties, or the \'etale Borel-Moore homology with $\ZZ/n\ZZ$-coefficient are examples of such a homology theory.
For a smooth variety $X$ and $A=\ZZ$ or $\ZZ/n\ZZ$, we have the following isomorphisms
\[
KH_0(X, A) \cong H_0(X, A),
\]
\[
KH_1(X, A) \cong H_1(X, A),
\]
and an exact sequence
\[
 \text{CH}_1(X)\otimes A \to H_2(X, A) \to KH_2(X, A) \to 0.
\]

Thus $KH_2(X, A)$ measures the failure of the integral Hodge/Tate conjecture. This motivates the following.

Consider the statements:
\begin{enumerate}

\item 
[$\textbf{S}(n, k)$](\textbf{S}mooth fibration): Let $X \to Y$ be a smooth projective morphism between smooth quasi-projective varieties such that a general fiber is rationally connected. Assume that $\dim X \leq n$. Then $$KH_a(X, A)\cong KH_{a}(Y, A)\text{ for } 0 \leq a \leq k$$ and $$KH_{k+1}(X, A) \to KH_{k+1}(Y, A) \text{ is surjective.}$$ 

\item 
[$\textbf{F}(n, k)$](\textbf{F}ibration): Let $X \to Y$ be a rationally connected fibration. Assume that $X$ and $Y$ are smooth and that  $\dim X \leq n$. Then $$KH_a(X)=KH_a(Y) \text{ for } 0 \leq a \leq k$$ and $$KH_{k+1}(X, A) \to KH_{k+1}(Y, A) \text{ is surjective.}$$
 Note that $\textbf{S}(n, k)$ is a special case of this.
 
\item 
[$\textbf{D}(n, k)$](\textbf{D}egeneration): Let $\pi: X \to Y$ be a rationally connected fibration. Assume that $\dim X\leq n$. Let $D \subset Y$ be a reduced simple normal crossing divisor such that $\pi^{-1}(D)$ is a simple normal crossing divisor, and let $E$ be the reduced part of $\pi^{-1}(D)$. Then we have $$KH_a(E, A)=KH_a(D, A), 0 \leq a \leq k$$ and $$KH_{k+1}(E, A) \to KH_{k+1}(D, A) \text{ is surjective.}$$

\item 
[$\textbf{R}(n, k)$](\textbf{R}esolution): Let $(X, \Delta)$ be a klt quasi-projective pair. Let $X' \to X$ be a resolution of singularities. Assume $\dim X \leq n$. Then $$KH_a(X', A)=KH_a(X, A) \text{ for } 0\leq a \leq k$$ and $$ KH_{k+1}(X', A) \to KH_{k+1}(X, A) \text{ is surjective.}$$
\end{enumerate}

Some special cases of the above statements are of particular interest:
\begin{enumerate}
\item[$\textbf{RC}(n, k)$](\textbf{R}ationally \textbf{C}onnected varieties): Let $X$ be a smooth projective rationally connected variety of dimension at most $n$. Then $KH_a(X, A)=0, 0<a\leq k, KH_0(X, A)=A$.
\item [$\textbf{DC}(n, k)$](\textbf{D}egeneration over a \textbf{C}urve): Let $X \to (B, b)$ be a rationally connected fibration over a pointed curve. Assume that $X$ is smooth of dimension $n$ and $X_b$ is a simple normal crossing divisor. Then $KH_a((X_b)_{\text{red}}, A)=0, 0<a\leq k, KH_0((X_b)_{\text{red}}, A)=A$.
\end{enumerate}

Needless to say, the first $3$ properties are  relevant for our problem. 
In this article we use the minimal model program to analyze these properties.
While the minimal model program improves the global properties of a variety, we have to pay the price of introducing singularities. 
It is interesting to note that klt singularities behave very much like a smooth point for many problems.
The property $\textbf{R}(n, k)$ is just another sample of this phenomenon.

In Section \ref{sec-MMP-KH}, we will prove the following result.
\begin{thm}[=Theorem \ref{thm:induction}]
We have the following implications.
\begin{enumerate}
\item
$\textbf{R}(n, k) +\textbf{F}(n, k)\implies \textbf{R}(n+1, k)$
\item
$\textbf{R}(n, k)+\textbf{F}(n, k) \implies \textbf{D}(n+1, k)$
\item 
$\textbf{RC}(n, k) \implies \textbf{S}(n, k)$ for $k \leq 2$.
\item 
$\textbf{D}(n, k)+\textbf{S}(n, k)\implies \textbf{F}(n, k)$
\end{enumerate}
\end{thm}

\begin{rem}
Totaro \cite{Totaro_IHC_Kodaira_dim_0} proved that an isolated $3$-fold rational singularity also satisfies the resolution property if one only considers the Kato homology $KH_2$. 
But even this could fail in higher dimensions. In fact, Ottem and Suzuki \cite{Ottem_IHC_Pencil_Enriques} constructed an example of a smooth complex $3$-fold with $H^i(X, \OO_X)=0, i>0$ for which the integral Hodge conjecture fails. 
Thus by taking a cone over $X$, one can construct an isolated $4$-fold rational singularity which does not satisfy the resolution property for $KH_2$. 
It seems that klt singularities is the natural class of singularities to consider the resolution property for general Kato homology.
\end{rem}

\begin{rem}
The proof of Theorem \ref{thm:induction}, especially the key lemma \ref{lem:runningMMP}, is a natural continuation of the study of homotopy types of the dual complex by de Fernex-Koll\'ar-Xu (\cite{deFKollarXuDualComplex}). 
In fact, once we know the dual complex is simply connected, then the contractibility is equivalent to vanishing of the higher homology of the complex $C^\bullet(\Delta^{=1}, 0, A)$ to be defined in Section \ref{sec-MMP-KH}.
\end{rem}

We finally summarize what is known about these statements. By the work of Bloch-Srinivas \cite{BlochSrinivas}, $KH_3(X, A)=0$ for a uniruled $3$-folds. The work of Voisin (\cite{VoisinIHC3fold}, see also Theorem \ref{thm:IHC}) implies that $KH_2(X, A)=0$ for a uniruled $3$-folds. Combine these two results with Theorem \ref{thm:induction}, we have the following.
\begin{cor}
The statements $\textbf{F}(3, k), \textbf{R}(4, k), \textbf{D}(4, k)$ hold.
\end{cor}

The article is organized in the following way. We review the concept of homology theory and Kato homology in Section \ref{sec-H-KH}, the integral version of the Hodge/Tate conjecture in Section \ref{sec-IHC}, some results in the minimal model program in Section \ref{sec-MMP}. 
In Section \ref{sec-MMP-KH} we study the behavior of Kato homology in the running of the minimal model program. 
We work out an explicit example (Example \ref{ex:smooth_blowup}) of how the Kato homology changes in the case of a smooth blow-up at the beginning of this section. 
Readers not so familiar with the minimal model program can read this example first to get a rough idea of the general situation.
Finally, the main theorems are proved in Section \ref{sec-main}.

%%%%%%%%%%%%%%%%%%%%%%%%%%%%%%%%%%%%%%%%%%%%%%%%%%%%%%%%%
%%
%%Acknowledgement and grant information
%%
%%%%%%%%%%%%%%%%%%%%%%%%%%%%%%%%%%%%%%%%%%%%%%%%%%%%%%%%%

\textbf{Acknowledgment:} I am grateful to Olivier Wittenberg for bringing this problem to my attention, for helpful email discussions, and for finding a gap in a preliminary draft. I also thank Jean-Louis Colliot-Th\'el\`ene for many comments, and Chenyang Xu for providing references. This work is partially supported by NSFC grants No. 11871155, No. 11831013, No. 11890660, No.11890662.

\section{Homology theory, Kato homology}\label{sec-H-KH}
The main reference for this section is \cite{KertzSaito, BlochOgus}.
\subsection{Homology theory}\label{subsec-homology}
Let $k$ be an algebraically closed field of characteristic $p\geq 0$. 

We denote by $\mathcal{C}$ the category of schemes $X/k$ that are separated and of finite type.
\begin{defn}
Let $\mathcal{C}_*$ be the category with the same objects as $\mathcal{C}$, but where morphisms are just the proper morphisms. Let $\text{Ab}$ be the category of abelian groups. A \textbf{homology theory} $H = \{H_a \}, a \in \ZZ$ on $\mathcal{C}$ is a sequence of covariant functors:
$H_a(\bullet): \mathcal{C}_* \to \text{Ab}$ satisfying the following conditions:
\begin{enumerate}
\item 
 For each open immersion $j : V\to  X$ in $\mathcal{C}$, there is a map $j^*: H_a(X) \to H_a(V)$, associated to $j$ in a functorial way.
\item \label{item 2}
If $i:Y\to X$ is a closed immersion in $X$, with open complement $j:V\to X$, there is a long exact sequence (called the \textbf{localization sequence})
\[
\ldots \xrightarrow{\partial} H_a(Y) \xrightarrow{i_*} H_a(X) \xrightarrow{j^*} H_a(V) \xrightarrow{\partial} H_{a-1}(Y) \to \ldots .
\]
The maps $\partial$ are called the \textbf{connecting morphisms}. This sequence is functorial with respect to proper morphisms or open immersions.

\end{enumerate}
A \textbf{morphism between homology theories} $H$ and $H'$ is a morphism $\Phi : H \to H'$ of functors on $\mathcal{C}_*$, which is compatible with the long exact sequences in (\ref{item 2}).
\end{defn}

\begin{ex}\label{ex:homology}
There are two type of homology theories that we will mainly use in this article.
\begin{enumerate}
\item \cite[Example 2.3]{BlochOgus} For a complex scheme $X$, we take the associated complex analytic space $|X|$ and take $H_a(X, \ZZ)=H_a^{BM}(|X|, \ZZ)$, the Borel-Moore homology. For $X$ smooth of pure dimension $d$, $H_a(X, \ZZ)=H^{2d-a}(|X|, \ZZ)$, the singular cohomology. For $X$ projective, $H_a(X, \ZZ)=H_a(|X|, \ZZ)$, the singular homology.

\item \cite[Example 1.4, Lemma 1.5]{KertzSaito} For schemes defined over an algebraically closed field $k$, and a torsion group $A$ whose elements are annihilated by a number relatively prime to the characteristic, we use the \'etale Borel-Moore homology
\[
H_a(X, A)=H^{-a}(X_{\text{\'et}}, Rf^! A),
\]
 where $f: X \to k$ is the structure map. In particular, for $X$ smooth of pure dimension $d$, $H_a(X, A)=H^{2d-a}_{\text{\'et}}(X, A(d))$.
\end{enumerate}
\end{ex}
\subsection{Kato homology}\label{subsec-Kato}
Let us fix an abelian group $A$ and a homology theory $H_\bullet(-)$ with coefficient in $A$. 
Denote by $X_{(i)}$ the set of points of $X$ whose Zariski closure is a subvariety of dimension $i$. 
For $x \in X_{(i)}$,  we define $$H_{a}(x, A)=\lim_{V \subset \bar{x}} H_{a}(V, A)$$ using the pull-back on homology, where the limit is taken over all Zariski open subsets of the closure of $x$.
Given such a homology theory, we have the spectral sequence of homological type associated to every $X \in  \text{Ob}(\mathcal{C})$, called the niveau spectral sequence (cf. \cite{BlochOgus}):
\[
E^1_{a, b}=\oplus_{x \in X_{(a)}} H_{a+b}(x, A) \implies H_{a+b}(X, A).
\]
\begin{defn}\label{def:KH}
For every $d$-dimensional scheme $X$ in $\mathcal{C}$, denote by $KH_a(X, A)$ the $a$-th Kato homology defined as the $a$-th homology of the complex
\[
 KH(X): E^1_{d, 0} \to E^1_{d-1, 0} \to \ldots \to E^1_{0, 0},
\]
where the differentials are induced by the connecting map $\partial$ in the localization exact sequence.
\end{defn}
\begin{rem}
The original definition in Kato \cite{Kato_KH} uses a different map between each terms of the complex than the one used by Bloch-Ogus \cite{BlochOgus}. It is shown in \cite[Paragraph 3.5, Theorem 3.5.1]{JSS_duality} that the two definitions only differ by a sign.
\end{rem}
The properties we need are the following:
\begin{enumerate}
\item[P1] $KH_a(X)$ is covariant with respect to proper morphisms and contravariant with respect to open immersions.
\item[P2]  Let $Z \subset X$ be a closed subvariety and $U$ its open complement. Then we have a long exact sequence
\[
\ldots \to KH_a(Z) \to KH_a(X) \to KH_a(U) \to KH_{a-1}(Z)\to \ldots.
\]
\item[P3] Let $f: Y \to X$ be a birational projective morphism, $E \subset Y$ the exceptional locus, and $B\subset X$ the image of $E$. Then we have a long exact sequence:
\[
\ldots \to KH_a(E) \to KH_a(Y) \oplus KH_a(B) \to KH_a(X) \to KH_{a-1}(E) \to \ldots.
\]
\item[P4]\label{SNC}  Let $D$ be a simple normal crossing variety. Then there is a spectral sequence $$\{E^1_{p, q}=KH_q(E_{p}, A), d_1: E^1_{p, q} \to E^1_{p-1, q}\}$$ converging to $KH_{q+p}(E, A)$, where $E_i, i\geq 0$ is the disjoint union of all the intersections of $(i+1)$ irreducible components of $E$, and $d_1: E^1_{p, q} \to E^1_{p-1, q}$ is the natural map induced by the inclusion.
\item[P5] More generally, for any simplicial scheme,
\[
D_\bullet=\{D_n (n\in \ZZ_{\geq 0}), \delta_a: D_n \to D_{n-1} (0\leq a \leq n)\}
\]
the Kato homology of $D_\bullet$ is defined by the homology of the double complex
\begin{equation}\label{doublecomplex}
\ldots \xrightarrow{\partial} KH(D_n) \xrightarrow{\partial} KH(D_{n-1}) \xrightarrow{\partial} \ldots \xrightarrow{\partial} KH(D_0),
\end{equation}
where $\partial=\sum_{a=0}^n (-1)^a (\delta_a)_*$. 
For the case of simple normal crossing divisor $D$ in a variety, there is a natural associated simplicial scheme, and the spectral sequence coming from the double complex is the one in (\ref{SNC}) and computes the Kato homology of $D$. The same construction applies to the case of an effective divisor $D$ in a variety $X$.

\item[P6]  If $X$ and $Y$ are smooth projective varieties stably birational to each other, then $KH_a(X) \cong KH_a(Y)$. More generally, if $Y \to X$ is a birational projective morphism between quasi-projective varieties, then $KH_a(Y) \cong KH_a(X)$ for all $a$.

\end{enumerate}

The first five properties follows directly from definition and properties of homology theory and Kato homology. 
For the stable birational invariance in characteristic $0$, by weak factorization theorem, 
it suffices to check the invariance for blow-up/blow-down and product with a projective space, 
which is easily computable by definition of Kato homology (use lines in projective spaces to realize homotopy between complexes).

For the birational invariance between quasi-projective varieties, first note that one can find a sequence of blow-up along smooth centers
\[
X_n\to X_{n-1} \to \ldots X_1 \to X_0=X
\] such that $X_n$ admits a morphism to $Y$. By the computation for blow-ups, we know that the composition $KH_a(X_n) \to KH_a(Y) \to KH_a(X)$ induces an isomorphism $KH_a(X_n) \cong KH_a(X)$. Thus $KH_a(Y) \to KH_a(X)$ is surjective, and this is true for any projective birational morphism. In particular, $KH_a(X_n) \to KH_a(Y)$ is also surjective. Then it follows that $KH_a(Y) \to KH_a(X)$ is injective.

\section{Integral Hodge/Tate conjecture for one cycles}\label{sec-IHC}
Voisin has made the following conjecture.
\begin{conj}\label{conj:IHC_RC}
 Let $X$ be a smooth projective separably rationally connected variety of dimension $d$ defined over an algebraically closed field. Then $H^{2d-2}_{\text{\'et}}(X, \ZZ_l(d-1))$ is algebraic. If $X$ is defined over the complex numbers, then $H^{2d-2}_{\text{sing}}(X, \ZZ)$ is algebraic.
 \end{conj}

 Usually this conjecture is referred as integral Hodge/Tate conjecture for one cycles on (separably) rationally connected varieties.
 
 \begin{rem}
 The usual formulation of Tate/Hodge conjecture says that the cycle class map is surjective onto a subspace of the cohomology group that is either invariant under certain Galois group (Tate) or lies in the $(p,p)$-part in the Hodge decomposition (Hodge).
 For (separably) rationally connected varieties, because of the existence of a decomposition of the diagonal \`a la Bloch-Srinivas \cite{BlochSrinivas}, we know that every class in $H^{2d-2}_{\text{\'et}}(X, \QQ_l(d-1))$ or $H^{2d-2}_{\text{sing}}(X, \QQ)$ is algebraic.
 Thus in the integral version, we ask whether the whole space of integral cohomology is generated by classes of algebraic cycles.
 \end{rem}
 
 \begin{thm}\label{thm:IHC}
 Conjecture \ref{conj:IHC_RC} is true for rationally connected smooth projective varieties defined in an algebraically closed field of charateristic $0$ in the following cases:
 \begin{enumerate}
 \item 
(Voisin \cite{VoisinIHC3fold}) $X$ has dimension $3$. 
\item (H\"oring-Voisin \cite{HoeringVoisin}) $X$ is a \textbf{smooth} Fano $4$-fold. 
\end{enumerate}  
 \end{thm}
 
 The strongest evidence for Conjecture \ref{conj:IHC_RC} comes from the following. First recall the Tate conjecture for varieties defined over a finite field.
 
 \begin{conj}[Tate conjecture]\label{conj:Tate}
 Let $X$ be a smooth projective variety of dimension $d$ defined over a finite field. Then the cycle class map
 \[
 CH^i(X) \otimes \QQ_l \to H^{2i}_{\text{\'et}}(X, \QQ_l(i))
 \]
 is surjective.
 \end{conj}
 
 Schoen \cite{SchoenIntegralTate} proved that the classical version of the Tate conjecture for divisor classes on surfaces (with $\QQ_l$-coefficients) implies an integral version of the Tate conjecture for curve classes on all smooth projective varieties over an algebraic closure of a finite field.
 Using deformation theory of rational curves and the above mentioned result of Schoen, Voisin proved the following:
 \begin{thm}[\cite{VoisinCurveClassRC}]\label{thm:Voisin}
 If the Tate conjecture (Conjecture \ref{conj:Tate}) holds for divisor classes in surfaces, then Conjecture \ref{conj:IHC_RC} holds.
 \end{thm}

The failure of the integral Hodge/Tate conjecture for one cycles is captured by Kato homology.

For the two homology theories in Example \ref{ex:homology} (and others), Bloch and Ogus prove the following theorem:
\begin{thm}\cite{BlochOgus}
Let $X$ be a smooth variety of dimension $d$ over an algebraically closed field $k$ and let $\mathcal{H}_n$ be the Zariski sheaf associated with the presheaf $U \mapsto H_j(U), 0 \leq j \leq d$, where $H_j(\bullet)$ is a homology theory.
Then there is a resolution by flasque sheaves:
\[
0 \to \mathcal{H}_j \to \oplus_{x \in X_{(d)}} {i_x}_* H_j(x) \to \oplus_{x \in X_{(d-1)}} {i_x}_* H_{j-1}(x) \to \ldots \to \oplus_{x \in X_{(d-j)}} {i_x}_* H_0(x).
\]
\end{thm}
This theorem implies an analogue of the Bloch formula for Chow groups.
If $X$ is a smooth variety defined over an algebraically closed field and we take homology with $\ZZ/n\ZZ$-coefficient,
\[
CH_i(X)/n \cong H_0(X, \mathcal{H}_{d-i}).
\]
If $X$ is a complex variety and we use the Borel-Moore homology with $\ZZ$-coefficient,
\[
CH_i(X)/\text{alg} \cong H_0(X, \mathcal{H}_{d-i}).
\]
By the above theorem and the Leray spectral sequence computing $H_n(X, A)$, we have the following isomorphisms for a smooth variety $X$:
\[
KH_0(X, A) \cong H_0(X, A),
\]
\[
KH_1(X, A) \cong H_1(X, A),
\]
and the long exact sequence
\[
H_3(X, A) \to KH_3(X,  A) \to(\text{CH}_1(X)/\text{alg}) \otimes A \to H_2(X, A) \to KH_2(X, A) \to 0.
\]
 
 In particular, when $X$ is a smooth projective complex variety and we take the Borel-Moore homology with $\ZZ$-coefficients, the torsion of $KH_2(X)$ measures the failure of the integral Hodge conjecture for one cycles on $X$. In general, for any smooth projective variety $X$ over an algebraically closed field, if $KH_2(X, A)=0$, then every class of $H_2(X, A)$ is algebraic. For later application, we need to generalize this to the case of a simple normal crossing variety.

\begin{lem}\label{lem:KatoSNC}
If $KH_2(D, A)=0$ for a simple normal crossing projective variety $D=\cup_{i \in I} D_i$, then $H_2(D, A)$ is generated by classes of algebraic curves.
\end{lem}

\begin{proof}
We have a map $H_a(X, A) \to KH_a(X, A)$ that is compatible with respect to proper push-forward and pull-back via open immersions.
Thus we have a commutative diagram of spectral sequences computing $KH_a(D, A)$ and $H_a(D, A)$ induced by these maps.

Since $KH_a(D_i, A)\cong H_a(D_i, A)$ for $a=0, 1$, in the $E^1$ page, $\{E^1_{p, a}\}$ of the spectral sequence abutting to $KH(D, A)$ and $H(D, A)$ are the same for $a=0, 1$. 
Thus for any $p$, $E^2_{p, 0}, E^2_{p, 1}, E^3_{p, 0}, E^3_{1, 1}$ are the same for both spectral sequences. 
As a result, $E^\infty_{1, 1}, E^\infty_{2, 0}$ are the same for both spectral sequences.
The assumption $KH_2(D, A)=0$ implies that $E^\infty_{1, 1}=E^\infty_{2, 0}=0$ for both spectral sequences and $E^\infty_{0, 2}=0$ for the spectral sequence computing $KH(D, A)$. 
Then a diagram chasing taking into account of the compatibility proves the statement.
\end{proof}

\section{Minimal model program}\label{sec-MMP}
A crucial ingredient in the proof of Theorem \ref{thm:Laurant} is the minimal model program (MMP). Here we recall some basic definitions and results that are used in the proof. 

Let $X$ be a normal variety over an algebraically closed field of characteristic $0$, and let $\Delta$ be
an effective $\QQ$-Weil divisor on X such that $K_X +\Delta$ is $\QQ$-Cartier. Let $f : Y \to X$ be a log resolution of
$(X, \Delta)$. This means that
\begin{enumerate}
\item $Y$ is smooth and $f$ is proper, birational.
\item Let $\{E_i\}$ be the set whose elements are all components of the exceptional divisors of $f$
and all components of the strict transform of $\Delta$, then the $E_i$
form a simple normal crossing
divisor.
\end{enumerate}

\begin{defn}
If we write 
\[
K_Y=f^*(K_X+\Delta)+\sum a_i E_i,
\]
the number $a_i$ is call \emph{discrepancy} of $E_i$ with respect to $(X, \Delta)$, denoted by $a_i=a(E_i, X, \Delta)$.

The pair $(X, \Delta)$ is \emph{log canonical (lc)} (resp. \emph{Kawamata log terminal (klt)}) if $a(E_i, X, \Delta)\geq -1$ (resp. $>-1$) for all $E_i$.
\end{defn}

\begin{defn}\label{def:dlt}
The pair $(X, \Delta)$ is \emph{divisorial log terminal (dlt)}, if $\Delta=\sum a_i D_i$ with coefficients $0< a_i \leq 1$, and if there is a closed set $Z$ such that
\begin{enumerate}
\item $X$ is smooth away from $Z$ and $\Delta|_{X-Z}$ is a simple normal crossing divisor.
\item For any birational morphism $f: Y \to X$, and any divisor $E \subset Y$ whose center is contained in $Z$, $a(E, X, \Delta)>-1$.
\end{enumerate}
\end{defn}

\begin{ex}
Let $X$ be a smooth variety and $\Delta=\sum a_i D_i, 0<a_i \leq 1$ be an effective $\QQ$-divisor whose support is simple normal crossing. Then $(X, \Delta)$ is a dlt pair (also an lc pair).
If $0 < a_i <1$, then $(X, \Delta)$ is klt.
\end{ex}

The property of being lc/klt/dlt is preserved under MMP (\cite[Corollary 3.43, 3.44]{KM98}). 
Thus dlt pairs comes naturally when one runs the MMP starting from a smooth variety with simple normal crossing divisors and it is good for inductions.
The following proposition justifies the occasional change of assumptions from dlt to klt (e.g. in $\textbf{R}(n, k)$).

\begin{prop}\cite[Proposition 2.43]{KM98}\label{prop:dltklt}
Let $(X, \Delta)$ be a quasi-projective dlt pair and $H$ an ample divisor. 
Let $\Delta_1$ be an effective divisor $\QQ$-Weil divisor such that $\Delta-\Delta_1$ is effective. 
Then there is a rational number $c>0$ and an effective $\QQ$-divisor $D$ such that $D \equiv_{\text{num}} \Delta_1+cH$ and $(X, \Delta-\epsilon\Delta_1+\epsilon D)$ is klt for all sufficiently small rational number $\epsilon$.
\end{prop}

\begin{defn}
Let $(X, \Delta)$ be an lc pair. A subvariety $Z \subset X$ is an \emph{lc center} if there is a proper birational morphism  $f: Y \to X$, a prime divisor $E\subset Y$  with discrepancy $a(E, X, \Delta)=-1$, and $Z=f(E)$.
\end{defn}

For a smooth variety $X$ with a simple normal crossing boundary divisor $D= \sum_i a_i D_i (0< a_i \leq 1)$, the lc centers are precisely the irreducible components of intersections of those $D_i$'s with coefficient equal to $1$.
 We have the following result for the description of lc centers on a dlt pair, which generalizes this fact.
\begin{thm}[{\cite[Proposition 3.9.2]{FujinoLogTerminal}}]\label{thm:Fujino}
Let $(X, D)$ be a dlt pair and $D_1, \ldots, D_r$ the irreducible components that appear with coeffiecient $1$.
\begin{enumerate}
\item The $s$-codimensional lc centers of $(X, \Delta)$ are exactly the irreducible components of the various intersections $D_{i_1} \cap \cdots D_{i_s}$ for $\{i_1, \ldots, i_s\} \subset \{1, \ldots, r\}$.
\item Every irreducible component of $D_{i_1} \cap \cdots D_{i_s}$ is normal and of pure codimension $s$.
\item Let $Z\subset X$ be any lc center. Assume that $D_i$ is $\QQ$-Cartier for some $i$ and $Z \not\subset D_i$. Then every irreducible component of $D_i|_Z$ is also $\QQ$-Cartier.
\item For each irreducible lc center $Z$, there is an effective $\QQ$-divisor $\Delta$ on $Z$ such that $(Z, \Delta+\sum_{Z \not\subset D_i} D_i)$ is dlt and $K_Z +\Delta \sim_\QQ (K_X+D)|_Z$.
\end{enumerate}
\end{thm}

\begin{thm}[Koll\'ar-Shokurov connectedness theorem]\label{thm:connected}
Let $(X, \Delta)$ be a dlt pair and let $W\subset X$ be an lc center. Let $f: X \to Y$ a birational proper morphism to a normal variety such that $-(K_X+\Delta)$ is $f$-ample. Assume that $W$ is contained in the exceptional locus of $f$. Denote by $Z$ the image of $W$. Then $f_*(\OO_W)=\OO_Z$. In particular, $Z$ is normal and $W \to Z$ has connected fibers.
\end{thm}
\begin{proof}
This follows the proof the connectedness theorem \cite[Theorem 5.48]{KM98}. Even though the statement in \emph{loc. cit.} is connectedness of fibers, the proof actually shows $f_*(\OO_W)=\OO_Z$.
\end{proof}

\begin{prop}\label{prop:dltbase}
Let $(X, \Delta) \to S$ be a quasi-projective dlt/klt pair over $S$ and $p: X \to Y$ be a contraction of $(K_X + \Delta)$-negative extremal ray over $S$. Let $Z$ be an irreducible lc center that is contained in the exceptional locus and $B \subset Y$ the image of $Z$. Then there is a boundary divisor $\Delta_B$ on $B$ such that $(B, \Delta_B)$ is dlt/klt.
\end{prop}

\begin{proof}
By Theorem \ref{thm:Fujino}, there is a divisor $\Delta_Z$ on $Z$ such that $K_Z+\Delta_Z\sim_\QQ K_X+\Delta$ and $(Z, \Delta_Z)$ is dlt. The image $B$ is normal by the connectedness theorem \ref{thm:connected}.  
Furthermore, the morphism $p|_Z: Z \to B$ is a $(K_Z+\Delta_Z)$-negative contraction. 
By the cone theorem \cite[Theorem 3.7]{KM98}, there is an ample $\QQ$-Cartier divisor $H$ on $X$ such that $K_X+\Delta_X+H$ is $\QQ$-linearly equivalent to the pull-back of an ample $\QQ$-Cartier divisor $H_Y$ on $Y$.

In the situation of a $(K_Z+\Delta_Z)$-negative contraction from a klt space, we know the base is also klt by \cite[Theorem 0.2]{Ambro_moduli_b}. For later applications in this article, this result already suffices since we may perturb $\Delta_Z$ so that $(Z, \Delta_Z)$ is klt.

In the dlt case, the proof follows the same lines as in the proof of \cite[Proposition 5.5]{XuHogadi}. Note that the assumption of $\QQ$-factoriality in \emph{loc. cit.} is not necessary once we know the existence of $H_Y$. 
\end{proof}

We need the following version of the existence of the Koll\'ar component (cf. \cite[Lemma 1, Remark 1]{XuFinitenessFundamentalGroup}).
\begin{lem}\label{lem:kollar_comp}
Let $(X, \Delta)$ be a quasi-projective klt pair and let $Z \subset X$ be a closed subvariety. There is a closed subset $W \subset Z$, a $\QQ$-divisor $H$ of $X$ and a $\QQ$-factorial normal variety $Y$ with birational projective morphism $f: Y \to U=X-W$ such that
\begin{enumerate}
\item 
There is a prime divisor $E \subset Y$ whose center in $U$ is $Z-W$, and $E$ is the only exceptional divisor for the morphism $f$. The divisors $-(K_Y+ f_*^{-1}(\Delta)+E), -E$ are relatively semiample. 
\item 
There is a divisor $\Delta_E$ on $E$ such that the pair $(E, \Delta_E)$ is klt. 
\item
The pair $(U, \Delta|_U+H|_U)$ is lc, and is klt away from $Z$. The divisor $E$ is the unique divisor such that $a(E, U, \Delta|_U+H|_U)=-1$.
\item The morphism $E \to Z$ is a rationally connected fibration.
\end{enumerate}
\end{lem}

\begin{proof}
This is essentially proved in \cite[Lemma 1]{XuFinitenessFundamentalGroup}. 
The only difference is that Xu assumes that $Z$ is a point and takes the log canonical model to make $-E$ relatively ample and $Y \to X$ an isomorphism away from $Z$ (which is a point). 
Here we would like to stop with a log minimal model and keep the $\QQ$-factoriality of $Y$, which is preserved by MMP when we start with a log resolution of $(X, \Delta)$. 
One can prove this lemma by following the argument in \emph{loc. cit.} almost word-by-word.
\end{proof}
\section{Minimal model program and Kato homology}\label{sec-MMP-KH}

In this section, we prepare ourselves for the proof of Theorem \ref{thm:induction} by studying the behavior of Kato homology in various situations in MMP.

Let $(X, \Delta)$ be a quasi-projective dlt pair. Denote by $\Delta^{=1}$ the union of irreducible components of $\Delta$ with coefficient $1$, $\{\Delta_i, i \in I=\{1, \ldots, m\} \}$ the irreducible components of $\Delta^{=1}$ , $\Delta_{i_0 \ldots i_d}=\cap_{i_0< \ldots <i_d \in I}\Delta_{i}$ and by $C^\bullet(\Delta^{=1}, k, A)$ the complex
\[
\oplus_{i \in I} KH_k(\Delta_i, A) \leftarrow \oplus_{i_0<i_1} KH_k(\Delta_{i_0 i_1}, A)\leftarrow \ldots \leftarrow \oplus_{i_0<\ldots<i_m} KH_k(\Delta_{i_0 \ldots i_m}, A)
\]

We first work out a simple example of blowing up along a smooth center.
In this example we allow a more general pair. 
\begin{defn}
Let $D$ be a $\QQ$-Weil divisor on a normal variety $X$. We say $D$ is a 
\textbf{subboundary divisor} if $D=\sum d_i D_i$ with $D_i$ the irreducible components of the support of $D$ and $d_i\leq 1$.
We say that $D$ is a \textbf{simple normal crossing subboundary divisor} if $D$ is a subboundary divisor, if $X$ is smooth, and if the support of $D$ is a simple normal crossing divisor.
\end{defn}
\begin{ex}\label{ex:smooth_blowup}
Let $(X, D)$ be a quasi-projective pair consisting of a smooth variety and a simple normal crossing subboundary divisor. Let $Z$ be a smooth subvariety of $X$ and $f: Y \to X$ the blow-up of $X$ along $Z$, and $E$ the exceptional divisor.  Denote by $D'$ the reduced part of the inverse image of $D^{=1}$. Assume that $Z$ intersects $D$ in a simple normal crossing way (that is, the union of $E$ and the strict transform of the support of $D$ is a simple normal crossing divisor). 

By the birational invariance and localization exact sequence, we immediately see that 
\[
KH_a(D') \cong KH_a(D^{=1}), a \geq 0.
\]

In the following, we would like to compare $C^\bullet(D^{=1}, k, A)$ and $C^\bullet(D', k, A)$, which would give more precise information for this isomorphism. 
These computations also motivates the proof of the key lemma \ref{lem:runningMMP}.

Case 1: $Z$ does not lie in $D^{=1}$. Then $E$ is not contained in $D'$, and the two complexes are simply the same by the birational invariance. 

Case 2: $Z$ is an irreducible component of $D_0 \cap \ldots \cap D_r, D_0, \ldots, D_r \subset D^{=1}$. Then $E$ is an irreducible component of $D'$. Write $D_i'$ for the strict transforms of $D_i$. 

Let $I$ be an index set such that $0 \not \in I$ and that $D_I \cap Z$ is non-empty. Consider an irreducible component $W$ of $E \cap D_I'$. It is a projective space bundle over its image in $X$. 

Note that $D_i'|_E$ is the relative $\OO(1)$ bundle for $E \to Z$ for $0 \leq i \leq r$.
 Thus one of the following two cases has to occur:
\begin{enumerate}
\item $W$ is isomorphic to its image in $X$ and
 $D_0' \cap W$ is empty.
 \item $D_0'\cap W$ is non-empty.
 \end{enumerate}
In the second case, both $D_0' \cap W$ and $W$ are projective space bundles over the same base, thus we have
\[
KH_a(W)\cong KH_a(D_0'\cap W).
\]
To sum up, let $W$ be an irreducible componet which belongs to the second case, and has minimal dimension among all such components.
Then
\[
KH_a(D_0' \cap W) \to KH_a(W)
\]
is a subcomplex of $C^\bullet(D', k, A)$ when placed in the right degree,
 which is quasi-isomorphic to $0$.
 
 We first take the quotient complex by all such subcomplexes. 
 Then we repeat this process. 
 Each time we take quotient by subcomplexes quasi-isomorphic to $0$, which come from irreducible components such that
 \begin{enumerate}
 \item  the irreducible component is contained in $E$;
 \item the irreducible component has non-empty intersection with $D_0'$;
 \item the irreducible component  has minimal dimension among all the remaining irreducible components that satisfy the above two properties.
\end{enumerate}  

The complex $C^\bullet(D', k, A)$ is quasi-isomorphic to the successive quotient by all the subcomplexes of the above form. 
The quotient complex is isomorphic to $C^\bullet(D^{=1}, k, A)$ by birational invariance of Kato homology.

Case 3: $Z$ is strictly contained in $D_0 \cap \ldots \cap D_r, D_0, \ldots, D_r \subset D^{=1}$. 
We use the same notations $D_i'$ as before. Then $D_0'|_E$ is the relative $\OO(1)$ for $E \to Z$.

For each index set $I$ such that $ D_I \subset Z$  and $0 \not \in I$,  the intersection $D_0 \cap D_I \cap Z$ is non-empty. 
Since $Z$ is not an irreducible component of $D_I$, there is an irreducible component $W$ of $D_I'$ such that $W \cap E$ is a positive dimensional projective space bundle over its image in $X$. 
Thus $D_0' \cap W \cap E$ is non-empty, and is a projective space bundle of one dimension less than $W \cap E$ over its image in $X$ (which is the same for $W\cap E$). 
Thus
\[
KH_a(W\cap E)\cong KH_a(D_0'\cap W \cap E), a \geq 0.
\]
As a result, if we take an irreducible component $W$ of $D_I'$ that has non-empty intersection with $E$, and that has the smallest dimension among all such components, then
\[
KH_a(D_0'\cap W\cap E)\to KH_a(W \cap E), a \geq 0
\]
is a subcomplex of $C^\bullet(D', k, A)$ when placed in the suitable degree. 
We take the quotient complex by all these subcomplexes. 
Then we successively take the quotient complex by subcomplexes of the form $KH_a(D_0'\cap W\cap E)\to KH_a(W \cap E), a \geq 0$ such that $W$ have the smallest dimension among all the remaining strata that is not contained in $D_0'$ and that has non-empty intersection with $E$.
In the end, we get a quotient complex that is quasi-isomorphic to $C^\bullet(D', k, A)$.
Moreover, the quotient complex is isomorphic to $C^\bullet(D, k, A)$ by birational invariance of Kato homology.
\end{ex}

\begin{rem}
We will generalize the argument in case 3 above in Lemma \ref{lem:runningMMP} to study the change of Kato homology under a divisorial contraction or a flip for dlt pairs. 
We interpret the intersection of divisors as lc centers, and use the resolution property $\textbf{R}(n, k)$ and the fibration property $\textbf{F}(n, k)$ to relate the Kato homology of different spaces. 
We have implicitly used some facts, (such as  $D_0' \cap W \cap E$ is irreducible and is an algebraic fibration over its image), which are clear in this example.
In the general case, such facts follow from the
Koll\'ar-Shokurov connectedness theorem \ref{thm:connected}.
\end{rem}

\begin{defn}\label{def:log-pullback}
Let $(X, \Delta)$ be a pair consisting of a normal variety and a subboundary divisor and $f: Y \to X$ a proper birational morphism. We use $K_Y+\Delta_Y\sim_\QQ f^*(K_X+\Delta), f_*(\Delta_Y)=\Delta_X$ to define the $\QQ$-divisor $\Delta_Y$, and call it the \emph{log pull-back} of $\Delta$.
\end{defn}
\begin{defn}\label{def:crepant}
Two pairs $(X_i, \Delta_i)$ consisting of a normal variety and a subboundary divisor are \emph{crepant birational equivalent} if there are proper birational morphisms $f_i: Y \to X_i$ such that the log pull-back of $\Delta_1$ equals the log pull-back of $\Delta_2$.
\end{defn}

Note that in the example above, in case 1 and 3, the exceptional divisor $E$ appears with coefficient strictly less than $1$ in the log pull-back of $D$. While in case 2, the log pull-back of $D$ is $D'$ (i.e. $E$ appears with coefficient precisely one). 
The dual complex of the divisors appearing with coeffiecient $1$ in the log pull-back of $D$ is a stellar subdivision in case $2$, and remains the same in the other two cases \cite[9.4]{deFKollarXuDualComplex}. 
Thus we have shown the following.

\begin{lem}\label{lem:logsmooth}
Let $(X, D)$ be a quasi-projective pair consisting of a smooth variety and a simple normal crossing subboundary divisor. Let $Z$ be a smooth subvariety of $X$ and $Y$ the blow-up of $X$ along $Z$, and $E$ the exceptional divisor.  Denote by $D'$ the log pull-back of $D$. Assume that $Z$ intersects $D$ in a simple normal crossing way. Then we have quasi-isomorphisms
\[
C^{\bullet}({D'^{=1}}, a, A) \rightarrow C^\bullet(D^{=1}, a, A), a \geq 0.
\]
\end{lem}

\begin{lem}\label{lem:crepant}
Let $(X_1, \Delta_1)$ and $(X_2, \Delta_2)$ be quasi-projective crepant birational dlt pairs, and let $(X_0, \Delta_0)$ be a dlt pair with morphism $p_i: (X_0, \Delta_0) \to (X_i, \Delta_i)$ such that $p_1^*(K_{X_1}+\Delta_1) = p_2^*(K_{X_2}+\Delta_2)=K_{X_0}+\Delta_0$. Assume that the property $\textbf{R}(\dim X_1-1, k)$ holds. Then we have quasi-isomorphisms
\[
C^{\bullet}(\Delta_1^{=1}, a, A) \xleftarrow{\text{qis}} C^{\bullet}(\Delta_0^{=1}, a, A) \xrightarrow{\text{qis}} C^\bullet(\Delta_2^{=1}, a, A), 0 \leq a \leq k.
\]
\end{lem}

\begin{proof}
By Definition \ref{def:dlt}, every dlt pair $(X, \Delta_X)$ has a log resolution $g: (Y, \Delta_Y) \to (X, \Delta_X)$ such that the dual complex of $\Delta_Y^{=1}$ is identified with $\Delta_X^{=1}$ and $K_Y+\Delta_Y=g^*(K_X+\Delta_X)$ (e.g. take the log resolution that is an isomorphism away from $Z$ in Definition \ref{def:dlt}).
Here the divisor $\Delta_Y$ is only a subboundary divisor.
Thus the assumption $R(\dim X_1-1, k)$ (without the surjectivity on $KH_{k+1}$) implies that we can compute the $E^2$ page of $KH_\bullet(\Delta_1^{=1})$ using the log resolution if we allow subboundary divisors.
Then the quasi-isomorphism follows from Lemma \ref{lem:logsmooth} and weak factorization.
\end{proof}

The following lemma is a generalization of case $3$ in Example \ref{ex:smooth_blowup}.
\begin{lem}\label{lem:runningMMP}
Let $\pi: (X, \Delta)\to S$ be a projective morphism to a quasi-projective variety $S$ such that $(X, \Delta)$ is dlt, 
and let $f: (X, \Delta) \dashrightarrow (X', \Delta')$ be a divisorial contraction of a $(K_X+\Delta)$-negative extremal ray or a flip over $S$, 
where $\Delta'=f_*(\Delta)$. 
Let $R$ be the extremal ray corresponding to the divisorial/flipping contraction. 
Assume that there is an irreducible component $\Delta_0$ of $\Delta^{=1}$ such that $\Delta_0 \cdot R>0$. 
Assume that $\textbf{R}(\dim X-1, k)$ and $\textbf{F}(\dim X-1, k)$ hold. 
Then there is a quasi-isomorphism 
\[
C^{\bullet}({\Delta}^{=1}, a, A) \rightarrow C^\bullet({\Delta'}^{=1}, a, A), 0 \leq a \leq k.
\]
As a consequence, in the $E^r (r\geq 2)$ page of the spectral sequence computing $KH_a(\Delta^{=1})$ and $KH_a({\Delta'}^{=1})$, we have isomorphisms
\[
E^r_{p, q}(\Delta^{=1})\cong E^r_{p, q}({\Delta'}^{=1}), 0 \leq q \leq k-r+2,
\]
\[
E^r_{p, q}(\Delta^{=1})\cong E^r_{p, q}({\Delta'}^{=1}), 0 \leq q \leq k, 0 \leq p \leq r-2
\]
and a commutative diagram for $0 \leq q \leq k-r+1$
\[
\begin{CD}
 E^r_{p, q}(\Delta^{=1}) @>d_r>> E^r_{p-r, q+r-1}(\Delta^{=1})\\
 @V\cong VV @V\cong VV\\
 E^r_{p, q}({\Delta'}^{=1}) @>d_r>> E^r_{p-r, q+r-1}({\Delta'}^{=1}), \\
 \end{CD}
\]
Furthermore, we have an isomorphism
\[
KH_a({\Delta}^{=1}, A) \cong KH_a ({\Delta'}^{=1}, A), 0 \leq a \leq k.
\]
Finally, all the isomorphisms, commutative diagrams, etc. are compatible with respect to the morphism to $S$.
\end{lem}

\begin{proof}
\begin{enumerate}
\item 
 For a divisorial contraction or a flip $f: (X, \Delta) \dashrightarrow (X', \Delta')$, we know that $(X', \Delta')$ is also dlt.  
 
 For every lc center $Z \subset \Delta_I$, either $Z$ is contained in $Ex(f)$, in which case we say that it is \textbf{contracted}, or $f$ is an isomorphism at the generic point of $Z$, in which case $f_*(Z)$ is again an lc center. Since discrepancy strictly increases for a divisor whose center is contained in $Ex(f^{-1})$, we see that every lc center of $(X', \Delta')$ arises in this way. 

 For any $\Delta_J$, there exists $\QQ$-divisors $S_J$ on $\Delta_J$ such that the pair $(\Delta_J, \Delta-\sum_{j \in J}\Delta_j+S_J)$ is dlt, and $\Delta_i|_{\Delta_J} (i \notin J)$ is $\QQ$-Cartier (Theorem \ref{thm:Fujino}). 

\item \label{deFKX} By \cite[Proof of Theorem 19]{deFKollarXuDualComplex}, we have the following description of contracted lc centers.  Let $Z$ be a contracted lc center that is not contained in $\Delta_0$. $Z$ is an irreducible component, as well as a connected component, of some $\Delta_I$. Then $Z_+:=\Delta_0\cap Z$ is an \emph{irreducible} contracted lc center. Conversely, for any contracted lc center $W$ contained in $D_0$, there is a contracted lc center $W_-$ that is not contained in $\Delta_0$ such that $W=\Delta_0 \cap W_-$. Furthermore, the operations $Z \mapsto Z_+, W \mapsto W_-$ are inverse to each other.

\item \label{iso}Let $Z$ be a contracted lc center that is not contained in $\Delta_0$. Since $\Delta_0 \cdot R>0$, the composition $g: \Delta_0 \cap Z \to Z \to B$ is dominant. Here $B$ is the image of $Z$ under the divisorial/flipping contraction. Koll\'ar-Shokurov connectedness theorem \ref{thm:connected} implies that $g: \Delta_0 \cap Z  \to B$ is an algebraic fiber space, that is, $g_*(\OO_{\Delta_0 \cap Z})\cong \OO_{B}$. Furthermore, both $\Delta_0 \cap Z \to B, Z \to B$ are rationally connected fibrations since both morphisms have dlt Fano general fibers.
Thus by the assumptions $\textbf{R}(\dim X-1, k)$ and $\textbf{F}(\dim X-1, k)$, we have 
\[
KH_a(\Delta_I \cap Z) \cong KH_a(Z), 0 \leq a \leq k.
\]

\item  In this step we will successively take quotient complexes of $C^\bullet(\Delta^{=1}, k, A)$ by a subcomplex quasi-isomorphic to the zero complex.
We start with the contracted lc centers that are of smallest dimension and that are not contained in $\Delta_0$. Let $Z\subset \Delta_I$ be one of such lc centers. Then the lc center $\Delta_0 \cap Z$ is also an irreducible contracted lc center by (\ref{deFKX}). By (\ref{iso}), and by the minimality condition
\[
KH_a(Z\cap \Delta_0) \to KH_a(Z), 0 \leq a \leq k,
\]
 when placed in suitable degree, is a subcomplex of $C^\bullet(\Delta^{=1}, a, A), 0 \leq a \leq k$. 
 In the first step, we quotient out such subcomplexes.
 Then we successively quotient out the subcomplexes coming from contracted lc centers that have minimal dimension among all the remaining contracted lc centers. 
 In the end, we get a quotient complex that only contains non-contracted lc centers and is quasi-isomoprhic to $C^\bullet(\Delta^{=1}, a, A)$ for each $0 \leq a \leq k$.
By the resolution property and birational invariance, this quotient complex is isomorphic to $C^\bullet({\Delta'}^{=1}, a, A)$ for each $0 \leq a \leq k$.

\item The isomorphisms in the $E^r$-page of the spectral sequence, the commutativity of the differential, and the isomorphism of Kato homology follow from the construction of the spectral sequence from the double complex (\ref{doublecomplex}).  
\end{enumerate}
\end{proof}

For the existence of $\Delta_0$ in Lemma \ref{lem:runningMMP} we need the following two lemmas.
\begin{lem}\cite[Lemma 21]{deFKollarXuDualComplex}\label{lem:Delta0}
Let $(X, \Delta)$ be a dlt pair and $g : X \to S $ a proper morphism. Assume that there is a numerically $g$-trivial effective divisor $A$ whose support equals $\Delta^{=1}$.
Let $f: X \dashrightarrow Y$ be a divisorial contraction or flip corresponding to a $(K_X + \Delta)$-negative extremal ray $R$ over $S$. Then
\begin{enumerate}
 \item
  either $Ex(f)$ does not contain any lc centers
\item
 or there is a prime divisor $\Delta_0\subset \Delta^{=1}$ such that $(\Delta_0 \cdot R)>0$.

\end{enumerate}
\end{lem}

\begin{lem}\cite[Lemma 23]{deFKollarXuDualComplex} \label{lem:contraction}
Let $X, Y$ be normal, $\QQ$-factorial varieties and $p : Y \to X$ a projective, birational morphism. Let $\Delta$ be a boundary on $Y$ such that $(Y, \Delta)$ is lc. Let $f : Y \dashrightarrow Y_1$ be a divisorial contraction or flip corresponding to a $(K_Y + \Delta)$-negative extremal ray $R$ over $X$. Then
\begin{enumerate}
\item either there is a divisor $E_R \subset Ex(p)$ such that $(E_R \cdot R) > 0$;
\item or $f$ contracts a divisor $E_f \subset Ex(p)$ and $Y_1 \to X$ is a local isomorphism at the generic point of $f(E_f)$.
\end{enumerate}
\end{lem}

Combine Lemma \ref{lem:runningMMP} and \ref{lem:Delta0}, we prove the following result.

\begin{cor}\label{cor:qis_mmp}
Let $(X, \Delta)$ be dlt and $g : X \to S $ a projective morphism. Assume that there is a numerically $g$-trivial effective divisor $A$ whose support equals $\Delta^{=1}$.
Let $f: X \dashrightarrow Y$be a birational map obtained by running an $(X, \Delta)$-MMP
over $S$. Set $\Delta_Y:= f_*\Delta$.  Then there are quasi-isomorphisms
\[
C^{\bullet}({\Delta}^{=1}, a, A) \rightarrow C^\bullet(\Delta_Y^{=1}, a, A), 0 \leq a \leq k.
\]
Furthermore, $KH_a({\Delta}^{=1})\cong KH_a({\Delta_Y}^{=1}), 0 \leq a \leq k$.
\end{cor}

The key technical theorem is the following.
\begin{thm}\label{thm:dlt}
Let $(X, \Delta)$ be a $\QQ$-factorial quasi-projective dlt pair, and let $f: Y \to X$ be a projective log resolution. Assume that the properties $\textbf{R}(\dim X-1, k)$ and $\textbf{F}(\dim X-1, k)$ hold. Let $E$ be the reduced part of the divisor $f^{-1}(\Delta^{=1})$.
Then we have a quasi-isomorphism of the complex in $E^1$-page
\[
C^{\bullet}(E, a, A) \rightarrow C^\bullet(\Delta^{=1}, a, A), 0 \leq a \leq k,
\]
As a consequence, in the $E^r (r\geq 2)$ page of the spectral sequence computing $KH_a(\Delta^{=1})$ and $KH_a({E})$, we have isomorphisms
\[
E^r_{p, q}(E)\cong E^r_{p, q}({\Delta}^{=1}), 0 \leq q \leq k-r+2,
\]
\[
E^r_{p, q}(E)\cong E^r_{p, q}({\Delta}^{=1}), 0 \leq q \leq k, 0 \leq p \leq r-2
\]
and a commutative diagram for $0 \leq q \leq k-r+1$
\[
\begin{CD}
 E^r_{p, q}(E) @>d_r>> E^r_{p-r, q+r-1}(E)\\
 @V\cong VV @V\cong VV\\
 E^r_{p, q}({\Delta}^{=1}) @>d_r>> E^r_{p-r, q+r-1}({\Delta}^{=1}), \\
 \end{CD}
\]
Furthermore, we have an isomorphism
\[
KH_a({E}, A) \cong KH_a ({\Delta}^{=1}, A), 0 \leq a \leq k,
\] and a surjection $KH_{k+1}(E, A) \to KH_{k+1}(\Delta^{=1}, A)$.
\end{thm}

\begin{proof}
Let $D=f_*^{-1}(\Delta^{<1})+E +\sum b_i F_i$ where the sum is taken over all $f$-exceptional divisors $F_i$ that are not contained in the support of $E$ and 
\[
b_i=\text{max}\{\frac{1}{2}, \frac{1-a(F_i, X, \Delta)}{2}\}.
\]
We run a $(K_Y+D)$-MMP over $X$ to arrive at $g: Y \dashrightarrow X'$ over $X$.

The $\QQ$-divisor $K_Y+D-f^*(K_X+\Delta)$ is effective and its support consists precisely of all the $f$-exceptional divisors with positive log discrepancy with respect to $(X, \Delta)$. Since the push-forward of $K_Y+D-f^*(K_X+\Delta)$ to $X'$ is effective, exceptional and semiample over $X$, it must be trivial by negativity. Therefore the map $g: Y \dashrightarrow X'$ contracts all components of $K_Y+D-f^*(K_X+\Delta)$. Thus the pair $(X', \Delta'=g_* D)$ is a dlt modification of $(X, \Delta)$.

Note that the condition about existence of a relatively trivial divisor in Lemma \ref{lem:Delta0} is always satisfied in our situation. 
Thus we have an isormophism of the $E^r$-page for $E$ and $\Delta'$ in the appropriate range as desired. 
By Lemma \ref{lem:crepant}, the $E^r$ page for ${\Delta'}^{=1}$ and $\Delta^{=1}$ agrees in the same range. 
This proves the isomorphism part of the statement.

By the surjectivity part for $KH_{k+1}$ in assumption $\textbf{R}(\dim X-1, k)$, for each irreducible component $E_i$ of $E$ that maps birationally onto its image $\Delta_i$, we have a surjection $KH_{k+1}(E_i, A) \to KH_{k+1}(\Delta_i, A)$. Thus the surjectivity part follows from functoriality of the spectral sequence and a simple diagram chasing of the differentials.
\end{proof}

\begin{cor}\label{cor:iso_resolution}
Let $X, Y$ be normal, $\QQ$-factorial varieties and $p : Y \to X$ a projective, birational morphism. 
Let $\Delta$ be the union of all the exceptional divisors of $p$ and assume that $(Y, \Delta)$ is dlt. 
Let $f : Y \dashrightarrow Y_1$ be a divisorial contraction or flip corresponding to a $(K_Y + \Delta)$-negative extremal ray $R$ over $X$. Assume the resolution property $\textbf{R}(\dim X-1, k)$ and fibration property $\textbf{F}(\dim X-1, k)$ hold. 
Set $\Delta_1=f_*(\Delta)$ and denote by $B$ (resp. $B_1$) the image of $\Delta$ (resp. $\Delta_1$ in $X$).
Then $KH_a(\Delta) \cong KH_a(B) $ if and only if $KH_a(\Delta_1) \cong KH_a(B_1)$ for $0\leq a \leq k$, and $KH_{k+1}( \Delta) \to KH_{k+1}(B)$ is surjective if  $KH_{k+1}( \Delta_1) \to KH_{k+1}(B_1)$ is surjective.
\end{cor}

\begin{proof}
By Lemma \ref{lem:contraction}, we only need to consider two cases.

Case 1, either $f$ is a flip, or $f$ contracts a divisor $D \subset \Delta(=Ex(p))$ and $Y_1 \to X$ is \textbf{NOT} a local isomorphism at the generic point of $f(D)$. 
Then there is a divisor $\Delta_0 \subset \Delta(=Ex(p))$ such that $(\Delta_0 \cdot R) > 0$.  Thus in this case $B=B_1$
We may choose a common log resolution $g: Z\to Y$ and $g_1: Z \to Y_1$. 
Note that $g^{-1}(\Delta)$ and $g_1^{-1}(\Delta_1)$ have the same support, since it is the union of all the exceptional divisors of $Z \to X$. Call it $E$. 
By Theorem \ref{thm:dlt}, we have isomorphisms $KH_a(E) \cong KH_a(\Delta) \cong KH_a(\Delta_1), 0 \leq a \leq k,$ and surjections $KH_{k+1} (E) \to KH_{k+1}(\Delta)$ and $KH_{k+1} (E) \to KH_{k+1}(\Delta_1)$. The statement then follows from the commutative diagram
\[
\begin{CD}
KH_{a} (E) @>>> KH_{a}(\Delta)\\
@VVV @VVV\\
KH_{a} (\Delta_1) @>>> KH_{a}(B).\\
\end{CD}
\]

Case 2, $f$ contracts a divisor $D \subset \Delta(=Ex(p))$ and $Y_1 \to X$ is a local isomorphism at the generic point of $f(D)$.
Note that by Proposition \ref{prop:dltbase}, the image of $D$ has dlt singularities.
In this case, $Y_1 \to X$ is not an isomorphism along $\Delta_1=\Delta-D$. 
By $\textbf{F}(\dim X-1, k)$, we have
\[
KH_a(\Delta-f^{-1}(\Delta_1)) \cong KH_a(B-B_1), 0 \leq a \leq k,
\]
and a surjection $KH_{k+1}(\Delta-f^{-1}(\Delta_1)) \to KH_{k+1}(B-B_1)$.
Then the assertion follows from the localization sequence for Kato homology.
\end{proof}

\section{Proof of main theorems}\label{sec-main}

We start with the following.
\begin{thm}\label{thm:induction}
We have the following implications.
\begin{enumerate}

\item\label{6.1}
$\textbf{R}(n, k) +\textbf{F}(n, k)\implies \textbf{R}(n+1, k)$
\item\label{6.2}
$\textbf{R}(n, k)+\textbf{F}(n, k) \implies \textbf{D}(n+1, k)$

\item \label{6.3}
$\textbf{RC}(n, k) \implies \textbf{S}(n, k)$ for $k \leq 2$.
\item \label{6.4}
$\textbf{D}(n, k)+\textbf{S}(n, k)\implies \textbf{F}(n, k)$

\end{enumerate}
\end{thm}

\begin{rem}
It is an intriguing question whether $\textbf{S}(n, k)$ is true for other $k$. 
For $k=3$, this amounts to compare the Griffiths group of one cycles on the total space of a smooth rationally connected fibration and that of the base. 
When the base is a curve, this is equivalent to proving that the degree $0$ part of the Chow group of zero cycles of the generic fiber is trivial (assuming $\textbf{RC}(n, 3)$).
In general, proving $\textbf{S}(n, k)$ requires proving some rather non-trivial statements comparing the (higher) Chow group of one cycles of the total space and the base of a rationally connected fibration.
\end{rem}

\begin{proof}
[Proof of Theorem \ref{thm:induction}, (\ref{6.1})]

We first consider the case when $X$ is $\QQ$-factorial.
Let us take a log resolution of $(X, \Delta)$, $p: X_0 \to X$ and denote by $E$ the reduced part of the exceptional divisor.  We run the MMP for $K_{X_0}+p_*^{-1}(\Delta)+E$ over $X$:
\[
(X_0, \Delta_0={p}_*^{-1}(\Delta)+E)\dashrightarrow (X_1, \Delta_1) \dashrightarrow \ldots \dashrightarrow (X_n, \Delta_n).
\]
Since $(X, \Delta)$ is klt, we have 
\[
K_{X_0}+{p}_*^{-1}(\Delta)+E=p^*(K_X+\Delta) +\sum (a(X, E_i)+1)E_i, a(X, E_i)+1>0,
\]
where the summation is taken over all the exceptional divisors.
By the negativity lemma, the end result of the MMP has to contract all the exceptional divisors. Since $X$ is $\QQ$-factorial, the exceptional locus of the morphism $X_n \to X$ has to be of codimension $1$. Therefore, $X_n \cong X$ (and $X_{n-1} \to X_n$ is a divisorial contraction).  Then we can apply Corollary \ref{cor:iso_resolution} to finish the proof.

In general, if $X$ is not $\QQ$-factorial, there is a closed subset $Z$ such that $U=X-Z$ satisfies the resolution property (e.g. take $Z$ to be the singular locus). Assume that $Z$ is non-empty. We apply Lemma \ref{lem:kollar_comp} to each irreducible component of $Z$. This produces a $\QQ$-factorial normal variety $Y$ and a birational projective morphism $f: Y \to X-W$ for some closed algebraic subset  $W \subset Z$ with one exceptional divisor $E_i$ for each irreducible component of $Z$. Up to enlarging $W$, we may assume the image of $E_i$ is dlt (even smooth). Then one can apply the localization sequence and properties $\textbf{R}(\dim X-1, k), \textbf{F}(\dim X-1, k)$ to conclude that $X-W$ satisfies  the resolution property. The statement then follows by noetherian induction.
\end{proof}

\begin{proof}[Proof of Theorem \ref{thm:induction}, (\ref{6.2})]

For the implication $\textbf{R}(n, k)+\textbf{F}(n, k) \implies \textbf{D}(n+1, k)$, 
first note that by the localization sequence, 
if $(X, E)$ and $(X', E')$ are smooth projective varieties with simple normal crossing boundary divisors such that there is a birational proper morphism $X-E \to X'-E'$, 
we have $KH_a(E) \cong KH_a(E')$. 
Also note that if $(X, F) \to (Y, E)$ and $(Y, E) \to (Z, D)$ are both rationally connected fibrations satisfying the simple normal crossing condition, 
then to prove the degeneration property for $(X, F) \to (Z, D)$, 
it suffices to show that the degeneration property holds for $(X, F) \to (Y, E)$ and $(Y, E) \to (Z, D)$ separately. 
Thus we only need to show the following:

Let $p: (X, E) \to (Y, D)$ be a rationally connected fibration with simple normal crossing boundary divisors $D$ and $E$ such that $p^{-1}(D)=E$.
Assume that $\textbf{R}(n, k), \textbf{F}(n, k)$ hold, 
and that any running of MMP for $(X, E)$ over $Y$ terminates with a Fano contraction $(X', E') \to Y$ of relative Picard number one. 
Then $KH_a(E, A)\cong KH_a(D, A), 0 \leq a \leq k$ and $KH_{k+1}(E, A) \to KH_{k+1}(D, A)$ is surjective.

By Corollary \ref{cor:qis_mmp}, we know that $KH_a(E) \cong KH_a(E'), 0 \leq a \leq k$. 
Since the relative Picard number of $X'/Y$ is $1$, each irreducible component $E_i'$ of $E'$ is the inverse image of an irreducible component $D_i$ of $D$ (\cite[Lemma 5.1, 5.2]{XuHogadi}). 
Since $\cap_{i \in I}E_i'$ and $\cap_{i \in I}D_i$ are precisely the log canonical centers of $(X', E')$ and $(Y, D)$, 
there is a one-to-one correspondence between the irreducible components of $\cap_{i \in I}E_i'$ and $\cap_{i \in I}D_i$, 
and the morphism between intersections $\cap_{i \in I}E_i' \to \cap_{i \in I} D_i$ is an algebraic fiber space (if non-empty) (\ref{thm:connected}). Moreover, the morphism $\cap_{i \in I}E_i' \to \cap_{i \in I} D_i$ has log Fano (hence rationally connected) generic fiber since it is the fiber of a $(K_{X'}+E')$-negative contraction. 
Thus by the fibration property $\textbf{F}(\dim X-1)$ and the resolution property $\textbf{R}(\dim X-1, k)$, we have
\[
KH_a(\cap_{i \in I}E_i', A) \cong KH_a(\cap_{i \in I}D_i, A), 0 \leq a \leq k,
\]
and a surjection $KH_{k+1}(\cap_{i \in I}E_i', A) \to KH_{k+1}(\cap_{i \in I}D_i, A)$.
Then the implication follows from the comparison of the spectral sequence.
\end{proof}

\begin{proof}
[Proof of Theorem \ref{thm:induction}, (\ref{6.3})]
The assertions $\textbf{RC}(n, k) \implies \textbf{S}(n, k)$ for $k \leq 2$ can be proved using induction on the dimension of the base. When the base is a point, the implication is trivial. For any positive dimensional smooth quasi-projective base $Y$, we can find a smooth ample divisor $H \subset Y$ such that the complement $U$ is affine. By induction hypothesis and the localization long exact sequence, we only need to consider the positive dimensional affine base case.  So in the following we assume that we have a rationally connected fibration $\pi : X \to Y$ over a smooth affine variety $Y$ of positive dimension. For any fiber of the rationally connected fibration $X_y$, we know that $H_1(X_y, A)=H^{2\dim X_y-1}(X_y, A)$ vanishes. Using the Leray spectral sequence for cohomology and the Artin vanishing for affine varieties, we deduce that $H_i(Y, A)\cong H_i(X, A)$ for $i=0, 1, 2$ and $H_3(X, A) \to H_3(Y, A)$ is surjective. This already shows that $KH_1(X, A) \cong KH_1(Y, A)$ and $KH_0(X, A) \cong KH_0(Y, A)$.
Recall that we have the following commuting diagram of long exact sequences:
\[
\begin{CD}
H_3(X) @>>>KH_3(X) @>>>CH_1(X) @>>> H_2(X)@>>>KH_2(X)@>>>0\\
@VVV @VVV @VVV @VV \cong V @VVV @VVV\\
H_3(Y) @>>>KH_3(Y)@ >>>CH_1(Y)@ >>> H_2(Y)@>>>KH_2(Y)@>>>0\\
\end{CD}
\]
We omit the coefficient $A$ in all the expressions above and $CH_1(X), CH_1(Y)$ are Chow groups of one cycles modulo algebraic equivalence with coefficient $A$. By the result of Graber-Harris-Starr \cite{GHS03}, $CH_1(X) \to CH_1(Y)$ is surjective. By a simple diagram chasing, we have the desired implications.
\end{proof}

\begin{proof}
[Proof of Theorem \ref{thm:induction}, (\ref{6.4})]
 
The assertion $\textbf{D}(n, k)+\textbf{S}(n, k)\implies \textbf{F}(n, k)$ follows directly from the localization sequence for Kato homology and a simple diagram chasing.
\end{proof}

Combining Theorems \ref{thm:IHC}, \ref{thm:Voisin}, \ref{thm:induction}, and Lemma \ref{lem:KatoSNC}, we have the following.
 
\begin{cor} \label{cor:IHC}
Assume all varieties are defined either over the complex numbers or an algebraically closed field of characteristic $0$. We use the homology theories in Example \ref{ex:homology}.
Let $\mcX \to (B, b)$ be a flat projetive morphism to a  pointed curve $B$. Assume that $\mcX$ is regular, that the generic fiber $X$ is rationally connected, that the fiber $\mcX_b$ is a simple normal crossing divisor, and
that one of the followings holds.
\begin{enumerate}
\item $\dim X \leq 3$.
\item $\dim X=d$ and the cycle class map $$CH_1(Y)/n \to H_2(Y, \ZZ/n\ZZ(-1))$$ is surjective for any smooth projective rationally connected variety $Y$ defined over $k$ of dimension at most $d$.
\item The Tate conjecture (Conjecture \ref{conj:Tate}) holds for all surfaces defined over a finite field.
\end{enumerate}
Then $KH_2((\mcX_b)_{\text{red}}, A)$ vanishes for $A\cong \ZZ$ (if $\mcX$ is defined over the complex numbers) or a torsion group. 
\end{cor}

Finally we prove the main theorem of the article.
\begin{thm} Let $k$ be an algebraically closed field of characteristic $0$, and let $X$ be a smooth projective rationally connected variety defined over $k\Semr{t}$. Assume that one of the followings holds.
\begin{enumerate}
\item $\dim X \leq 3$.
\item $\dim X=d$ and the cycle class map $$CH_1(Y)/n \to H^{2d-2}_{\text{\'et}}(Y, \ZZ/n\ZZ(d-1))$$ is surjective for any smooth projective rationally connected variety $Y$ defined over $k$ of dimension at most $d$.
\item The Tate conjecture (Conjecture \ref{conj:Tate}) holds for all surfaces defined over a finite field.
\end{enumerate}
Then the degree map induces an isomorphism $$\deg: CH_0(X)\cong \ZZ.$$
\end{thm}

\begin{proof}
For a rationally connected variety $X$ over any field, by the decomposition of the diagonal principle of Bloch-Srinivas, we know that there is a positive integer $N$ such that the kernel of the degree map is $N$-torsion. 
If $X$ is defined over $k\Semr{t}$, we know the degree map is surjective since $X$ has a $k\Semr{t}$-rational point (which is a corollary of the Greenberg's approximation theorem \cite{Greenberg} and the Graber-Harris-Starr theorem \cite{GHS03} for rationally connected fibrations over a curve).
Thus it suffices to show that the mod $N$ degree map $CH_0(X)/N \to \ZZ/N \ZZ$ is injective.
Since a smooth projective rationally connected variety$X$  over an algebraically closed field of characteristic $0$ has vanishing $H^{2\dim X-1}$, using Hochschild-Serre spectral sequence, we conclude that the mod $N$ degree map can be identified with the map to the \'etale cohomology:
\[
CH_0(X)/N \to H^{2d}_{\text{\'et}}(X, \ZZ/N\ZZ(d)).
\]
We choose a model $\mcX$ of $X$ over $\SP k\Sem{t}$ with regular total space and simple normal crossing closed fiber $\mcX_0$.
By the result of Saito-Sato \cite{SaitoSato_0_cycle}, we know that there is an isomorphism
\[
CH_1(\mcX)/N  \cong H^{2d}_{\text{\'et}}(\mcX, \ZZ/N\ZZ(d)).
\]
By the following commutative diagram
\[
\begin{CD}
CH_1((\mcX_0)_{\text{red}})/N @>>> CH_1(\mcX)/N @>>> CH_0(X)/N @>>>0\\
@VVV @V\cong VV @VVV \\
H^{2d}_{(\mcX_0)_{\text{red}}}(\mcX, \ZZ/N\ZZ(d))@>>>H^{2d}(\mcX, \ZZ/N\ZZ(d)) @>>> H^{2d}(X, \ZZ/N\ZZ(d)),
\end{CD}
\]
it suffices to show that $$CH_1((\mcX_0)_{\text{red}})/N \to H^{2d}_{(\mcX_0)_{\text{red}}}(\mcX, \ZZ/N\ZZ(d))\cong H_2((\mcX_0)_{\text{red}}, \ZZ/N \ZZ(-1))$$ is surjective.

We do not know MMP over $\SP k\Sem{t}$. But we may approximate the model $\mcX \to \SP k\Sem{t}$ to arbitrary high order by a family over a pointed affine curve. In particular, we may assume the family over the affine curve has smooth total space and SNC special fiber. Then we may use Corollary \ref{cor:IHC} to finish the proof.
\end{proof}

%%%%%%%%%%%%%%%%%%%%%%%%%%%%%%%%%%%%%%%%%%%%%%%%%%%%%%%%%%%%%%%%

\end{document}